\documentclass[11pt]{amsart}
\usepackage[margin=2.5cm]{geometry}
\usepackage{mathrsfs}
\usepackage{epic,eepic,epsf,epsfig}
\usepackage{amsfonts,srcltx,mathrsfs}

\usepackage{multirow}
\usepackage{amsmath}
\usepackage{amssymb}
\usepackage{amsbsy}
\usepackage{graphicx}
\usepackage{amsfonts}%,srcltx}
\usepackage{color}

\numberwithin{equation}{section}
\newtheorem{lem}{Lemma}[section]%
\newtheorem{notation}{Notation}[section]%
\newtheorem{theorem}[lem]{Theorem}%
\newtheorem{problem}[lem]{Problem}%
\newtheorem{prop}[lem]{Proposition}%
\newcommand{\Sym}{\mathop{\mathrm{Sym}}}

\def\a{\alpha}

  \def\G{\Gamma}

\def\nd{\mathrel{\bigm|\kern-.7em/}}

\def\f{\noindent}

\def\Aut{\hbox{\rm Aut}}
\def\Cay{\hbox{\rm Cay}}

\def\Haar{\hbox{\rm Haar}}

\def\mz{{\mathbb Z}}

\def\norm#1#2{{\bf N}_{{#1}}{{(#2)}}}

\begin{document}
\title[On HDR]{On Haar digraphical representations of groups}

\author{Jia-Li Du}
\address{Jia-Li Du, Department of Mathematics, China University of Mining and Technology, Xuzhou 221116, China}
\email{dujl@cumt.edu.cn}

\author{Yan-Quan Feng}
\address{Yan-Quan Feng, Department of Mathematics, Beijing Jiaotong University, Beijing 100044, China}
\email{yqfeng@bjtu.edu.cn}

\author{Pablo Spiga}
\address{Pablo Spiga, Dipartimento di Matematica e Applicazioni, University of Milano-Bicocca, Via Cozzi 55, 20125 Milano, Italy} 
\email{pablo.spiga@unimib.it}

\footnotetext[1]{Corresponding author Jia-Li Du. E-mail: dujl@cumt.edu.cn}

\date{}
 \maketitle

\begin{abstract}
In this paper we extend the notion of digraphical regular representations in the context of Haar digraphs.  Given a group $G$, a {\em Haar digraph} $\G$ over $G$ is a bipartite digraph having a bipartition $\{X,Y\}$ such that $G$ is a group of automorphisms of $\G$ acting regularly on $X$ and on $Y$. We say that $G$ admits a {\em Haar digraphical representation} (HDR for short),
if there exists a Haar digraph over $G$ such that its automorphism group is isomorphic to $G$.
In this paper, we classify finite groups admitting a HDR.

\bigskip
\f {\bf Keywords:} Semiregular group, regular representation, DRR, GRR, Haar digraph.

\medskip
\f {\bf 2010 Mathematics Subject Classification:} 05C25, 20B25.
\end{abstract}

\section{Introduction}
By a {\em digraph} $\Gamma$, we mean an ordered pair $(V, A)$ where the vertex set $V$ is a non-empty set and the arc set $A \subseteq V \times V$ is
a binary relation on $V$. The elements of $V$ and $A$ are called vertices and arcs of $\Gamma$, respectively. For simplicity, we write $V(\Gamma):=V$ and $A(\Gamma):=A$.
An automorphism of $\Gamma$ is
a permutation $\sigma$ of $V$ fixing $A$ setwise, that is, $(x^\sigma , y^\sigma ) \in A$ for every $(x, y) \in A$. The
digraph $\Gamma $ is a {\em graph} if the binary relation $A$ is symmetric.

A digraph is called {\em regular} if each vertex has the same out-valency and the same in-valency.
Throughout this paper, all groups and digraphs are finite, and all digraphs are regular.

Let $G$ be a group and let $S$ be a subset of $G$.
The {\em Cayley digraph} $\G:=\Cay(G,R)$ is the digraph with $V(\Gamma):=G$ and with
$A(\G):=\{(g,rg) ~|\ g\in G,r\in R\}$.
The right regular representation of $G$ gives rise to an embedding of $G$ into $\Aut(\Gamma)$ and we identify $G$ with its image under this permutation representation. We say that a group admits a {\em (di)graphical
regular representation} (resp. GRR or DRR for short)
if there exists a Cayley (di)graph $\G$ over
$G$ such that $\Aut(\G)= G$.
Babai~\cite{Babai} proved that, except for $Q_8$, $\mz_2^2$,
$\mz_2^3$, $\mz_2^4$ and $\mz_3^2$, every finite group admits a DRR. It is clear that, if
a group $G$ admits a GRR, then $G$ admits a DRR, however
the converse is not true. Indeed, despite the natural argument used by Babai for the classification of groups admitting a DRR, the classification of groups admitting a GRR has required considerable more work. For some of the most influential papers along the way we refer to~\cite{Imrich,ImrichWatkins,NowitzWatkins1,NowitzWatkins2}.
Watkins~\cite{Watkins} observed that there are two infinite
families of graphs admitting no GRR:
generalised dicyclic groups, and abelian groups of exponent
greater than two. Then, Hetzel~\cite{Hetzel} has proved that besides these two infinite families, among soluble groups, there are only 13 more groups admitting no GRR. Finally, Godsil~\cite{Godsil} has put the last piece into the puzzle and has shown that every non-solvable group admits a GRR, and so completed the classification of groups admitting a GRR.

Once the classification of DRRs and GRRs was completed, researchers proposed and investigated various natural generalisations. For instance, Babai and Imrich~\cite{BabaiI} have classified finite groups admitting a tournament regular representation, TRR for short. Morris and Spiga~\cite{MorrisSpiga1,MorrisSpiga3,Spiga}, answering a question of Babai~\cite{Babai}, have classified the finite groups admitting an oriented regular representation, ORR for short.
For more results, generalising the classical DRR and GRR classification in various direction, we refer to
\cite{DSV,DFS1,DFS,MSV,Spiga2,Spiga3,Xiaf}.

We now describe the generalisation we intend to investigate in this paper. Let $G$ be a permutation group on a set $\Omega$ and let $\omega \in \Omega$. Denote by $G_\omega$ the stabilizer of $\omega$ in $G$, that is, the subgroup
of $G$ fixing $\omega$. We say that $G$ is {\em semiregular} on $\Omega$ if $G_\omega = 1$ for every $\omega \in \Omega$, and {\em regular} if it is semiregular and transitive.
An $m$-Cayley (di)graph $\Gamma$ over a group $G$ is defined as a (di)graph which has a semiregular group of automorphisms
isomorphic to $G$ with $m$ orbits on its vertex set. When $m = 1$, $1$-Cayley (di)graphs are the usual Cayley (di)graphs.
We say that a group $G$ admits a (di)graphical $m$-semiregular representation (D$m$SR and G$m$SR, for short),
if there exists a regular $m$-Cayley (di)graph $\Gamma$ over $G$ such that $\Aut(\Gamma)\cong
 G$. In particular, D$1$SRs and G$1$SRs are the
usual GRRs and DRRs. For each $m\in\mathbb{N}$, we have  classified in~\cite{DFS} the finite groups admitting a D$m$SR and the finite groups admitting a G$m$SR. In this paper we propose a natural variant of this problem.

A {\em bipartite} $2$-Cayley (di)graph (over a group $G$, where the two parts of the bipartition are the two orbits of $G$) is known as {\em Haar (di)graph} in the literature. We say that a finite group $G$ admits a {\em Haar (di)graphical representation} (resp. HDR or HGR for short),
if there exists a Haar (di)graph over $G$ such that its automorphism group isomorphic to $G$.

\begin{theorem}\label{theo=main}
With the only exceptions of $\mz_1$, $\mz_2$, $\mz_3$, $\mz_2^2$ and $\mz_2^3$, every finite group admits a $\mathrm{HDR}$. 
\end{theorem}

Du {\em et al}~\cite[Lemma 2.6(i)]{Duxu} have shown that Haar graphs over abelian groups are Cayley graphs. Hence, abelian groups do not admit HGRs.  Est\'elyi~\cite[Proposition~$11$]{Estelyi} has proved that the dihedral group of order $2n$ admits a HGR if and only if $n\ge 8$. To end this section, we  propose the following problem.

\begin{problem}\label{question}
Classify finite groups admitting a $\mathrm{HGR}$.
\end{problem}
We are not sure what the answer to this problem might be, but besides the finite abelian groups we are aware of no infinite family of groups admitting no HGR. For instance, every generalised quaternion group of order $4n$ with $4\le n\le 100$ admits a HGR. 

\section{Preliminaries and notation}

In what follows, we describe some preliminary results which will be used later.
We start by recalling  Babai's classification of DRRs.

\begin{theorem}\label{prop=DRR} {\rm \cite[Theorem 2.1]{Babai}}
A finite group $G$ admits a $\mathrm{DRR}$ if and only if $G$ is not isomorphic to one of the following five groups $Q_8$, $\mz_2^2$, $\mz_2^3$, $\mz_2^4$ or $\mz_3^2$.
\end{theorem}

We recall that a tournament is a digraph $\Gamma$ such that, for every two distinct vertices $x, y \in V(\Gamma)$, exactly one of
$(x, y)$ and $(y, x)$ is in $A(\Gamma)$. Observe that the Cayley digraph $\Cay(G,R)$ is a tournament if and only if $R\cap R^{-1}=\emptyset$
and $R\cup R^{-1}=G\setminus \{1\}$. In particular, finite groups of even order have no TRR. 

\begin{theorem}\label{prop=TRR} {\rm \cite[Theorem 1.5]{BabaiI}}
A  finite group of odd order admits a $\mathrm{TRR}$ if and only if it is not isomorphic to  $\mz_3^2$.
\end{theorem}

Let $G$ be a group. Consistently throughout the whole paper, for not making our notation too cumbersome to use,
we denote the element $(g, i)$ of the Cartesian product $G \times \{0,1\}$ simply by $g_i$. In particular, we write $G_0=G\times \{0\}=\{g_0\mid g\in G\}$ and $G_0=G\times \{1\}=\{g_1\mid g\in G\}$.

 Let  $S$ and $T$ be subsets of $G$. We define $$\Haar (G,S,T)$$ to be the digraph having vertex set $G\times\{0,1\}=G_0\cup G_1$ and having arc set the union of $\{(g_0, (sg)_1)~|~g\in G, s\in S\}$ and $\{(g_1, (tg)_0)~|~g\in G, t\in T\}$. 
Now, $G$ induces a subgroup of $\Aut(\Haar (G,S,T))$ by defining:
\begin{center}
$(h_i)^{g}= (hg)_i$,\,\,\, for every $g,h\in G$ and $i\in\{0,1\}$.
\end{center}
For not making the notation too cumbersome, we identify $G$ with this subgroup of $\Aut(\Haar(G,S,T))$.
Clearly, $G$ acts semiregularly with two orbits $G_0$ and $G_1$ on $V(\Haar (G,S,T))$. In particular, $\Haar (G,S,T)$ is a Haar digraph over $G$. It is not hard to see that every Haar digraph over $G$ is isomophic to $\Haar(G,S,T)$, for some suitable subsets $S$ and $T$ of $G$.

For every automorphism $\a$ of $G$
and for every $x,y\in G$, we define two permutations $\delta_{\a,x}$ and $\sigma_{\a,y}$ of $G_0\cup G_1$ by setting
\begin{align}\nonumber
\delta_{\a,x}&:
\begin{cases}g_0\mapsto (g^{\a})_0,&\forall g\in G,\\ 
g_1\mapsto (xg^{\a})_1,&\forall g\in G,\\
\end{cases}\\\label{eq:3}
\sigma_{\a,y}&:
\begin{cases}
g_0\mapsto (g^{\a})_1,&\forall g\in G,\\
g_1\mapsto (yg^{\a})_0,&\forall g\in G.
\end{cases}
\end{align}
The permutation $\delta_{\a,x}$ will play little role in this paper, but $\sigma_{\a,y}$ will be rather important. Then, we define
\begin{align*}
     X&:=\{\delta_{\a,x}~|~S^{\a}=x^{-1}S \textrm{ and }T^{\a}=Tx\},\\
     Y&:=\{\sigma_{\a,y}~|~~S^{\a}=y^{-1}T\textrm{ and }T^{\a}=Sy\}.
\end{align*}

We conclude this  section by reporting a result describing the normaliser in $\Aut(\Haar(G,S,T))$ of $G$. 

\begin{prop}\label{prop=normalizer} {{{\rm (\cite[Theorem 1]{Arezoomand} and \cite[Lemma 2.1]{Hujdurovic})}}}
Let $G$ be a finite group and let $S$ and $T$ be subsets of $G$, then $$\norm{\Aut(\mathrm{Haar}(G,S,T))}{G}=GL=\{g\ell\mid g\in G,\ell\in L\},$$ where $L=X\cup Y$ and $L\cap G=1$.
\end{prop}

\section{Proof of Theorem~\ref{theo=main}}

In this section, we prove Theorem~\ref{theo=main}.

\begin{lem}\label{lem=S=T}
Let $G$ be a finite group and let $S$ be a subset of $G$. The Haar digraph $\Haar(G,S,S)$
is vertex transitive and hence $\Haar(G,S,S)$ is not a $\mathrm{HDR}$.
\end{lem}
\begin{proof}
Let $\G:=\Haar(G,S,S)$ and let $\phi$ be the permutation of $V(\G)=G_0\cup G_1$ with $g_0\mapsto g_{1}$ and $g_1\mapsto g_0$,
for each $g\in G$. 

For every $g \in G$ and $s\in S$, $(g_0,(sg)_1)^{\phi}=(g_1,(sg)_0)$ and $(g_1,(sg)_0)^{\phi}=(g_0,(sg)_1)$
are arcs of $\G$ and hence $\phi$ is an automorphism of $\G$ interchanging $G_0$ and $G_1$.
As $G$ is transitive on $G_0$ and $G_1$, we deduce that $\langle G,\phi\rangle$ is transitive on $V(\G)$. Hence $\G$ is vertex transitive and $\G$ is not a HDR.
\end{proof}

\begin{notation}\label{notation:1}{\rm Let $G$ be a finite group and let $\phi \in \Sym(G)$ be a permutation of $G$.
We let $\phi'$ be the permutation of $G_0\cup G_1$ defined by $$(g_i)^{\phi'}=(g^{\phi})_i,\,
\textrm{  for each }g\in G \textrm{  and for each  } i\in\{0,1\}.$$}
\end{notation}

\begin{lem}\label{lem=automorphism}
Let $G$ be a finite group and let $\phi \in \Sym(G)$.
Then, $\phi'\in\Aut(\Haar(G,S,T))$
if and only if $\phi\in \Aut(\Cay(G,S))\cap\Aut(\Cay(G,T))$.
\end{lem}

\begin{proof}Let $\Sigma_1:=\Cay(G,S)$, $\Sigma_2:=\Cay(G,T)$ and $\G:=\Haar(G,S,T)$.
The permutation $\phi$ lies in $\Aut(\Cay(G,S))\cap\Aut(\Cay(G,T))$ if and only if $$(g,sg)^{\phi}=(g^{\phi},(sg)^{\phi})\in A(\Sigma_1) \textrm{ and }(g,tg)^{\phi}=(g^{\phi},(tg)^{\phi})\in A(\Sigma_2),$$  for each $g\in G$, $s\in S$ and $t\in T$. This happens if and only if, for each $s\in S$ and $t\in T$, there exist $s'\in S$ and $t'\in T$ with $$(sg)^{\phi}=s'g^{\phi} \textrm{ and }(tg)^{\phi}=t'g^{\phi}.$$
In turn, this happens if and only if $(g_0,(sg)_1)^{\phi'}=((g^{\phi})_0,((sg)^{\phi})_1)=((g^{\phi})_0,(s'g^{\phi})_1)\in A(\G)$
and $(g_1,(tg)_0)^{\phi'}=((g^{\phi})_1,((tg)^{\phi})_0)=((g^{\phi})_1,(t'g^{\phi})_0)\in A(\G)$,
that is, $\phi'\in \Aut(\G)$.
\end{proof}

\begin{lem}\label{lem=noDRR}
Let $G$ be a finite group admitting no $\mathrm{DRR}$. Then $G$ admits a $\mathrm{HDR}$ except when $G$ is isomorphic to either $\mz_2^2$ or $\mz_2^3$.
\end{lem}
\begin{proof} By Theorem~\ref{prop=DRR}, $G$ is isomorphic to one of the following groups: $Q_8$, $\mz_3^2$, $\mz_2^2$, $\mz_2^3$ or $\mz_2^4$. It can be verified with the computer algebra system \textsc{Magma}~\cite{magma} that $\mz_2^2$ and $\mz_2^3$ admit no HDR. %This also follows from~\cite[Theorem~$1.1$]{DFS}, because $\mz_2^2$ and $\mz_2^2$ admit no graphical $2$-semiregular representation

When $G=\langle a\rangle\times\langle b\rangle\times\langle c\rangle\times \langle d\rangle\cong \mz_2^4$, it can be verified with \textsc{Magma} that  $$\Haar(G,\{1,a,b,c,d,ab\},\{1,a,c,bd,abc,bcd\})$$ is a HDR.
Similarly, when $G=\langle a,b~|~a^4=b^4=1,b^2=a^2,a^b=a^{-1}\rangle\cong Q_8$,
 $$\Haar(G, \{1,a,b\},\{a^2,b^3,ab\})$$ is a HDR  and, when 
$G=\langle a\rangle\times\langle b\rangle\cong \mz_3^2$, $$\Haar(G,\{1,a,b\},\{a,b^2,ab\})$$ is a HDR.
\end{proof}

\begin{notation}\label{notation:2}
{\rm Let $\Gamma$ be a digraph and let $v$ be a vertex of $\Gamma$. We denote  by $\Gamma^+(v)$ and by $\Gamma^-(v)$ the out-neighbourhood and the in-neighbourhood of $v$ in $\Gamma$.}
\end{notation}

\begin{lem}\label{lem=A+}
Let $G$ be a finite group and let $R$ be a subset of $G$ with $\Cay(G,R)$ a $\mathrm{DRR}$ of $G$, $1\notin R$ and $|R|<|G|/2$.
Let $L$ be a subset of $G\setminus (R^{-1}\cup\{1\})$ with $|L|=|R|$
and let $\G:=\Haar(G,R\cup\{1\},L\cup\{1\})$.
Then 
\begin{enumerate}
\item\label{eqle:0}$\Gamma^+(g_i)\cap \Gamma^-(g_i)=\{g_{1-i}\}$,  for every
 $g\in G$ and for every $i\in\{0,1\}$,
\item\label{eqle:1} $|\Aut(\G):G|\le 2$,  
\item\label{eqle:2} $\G$ is a $\mathrm{HDR}$  if and only if $R^{\a}\neq L$ for each $\a\in \Aut(G)$, and
\item\label{eqle:3} the subgroup of $\Aut(\G)$ fixing $G_0$ and $G_1$ setwise is $G$.
\end{enumerate}
\end{lem}
\begin{proof} From the definition of the arc set of $\Haar(G,R\cup\{1\},L\cup\{1\})$, for every $g\in G$, we have
\begin{align*}
\Gamma^+(g_0)&=(Rg\cup \{g\})_1=\{(rg)_1~|~r\in R\cup\{1\}\},\\
\Gamma^-(g_0)&=(L^{-1}g\cup \{g\})_1=\{(l^{-1}g)_1~|~l\in L\cup\{1\}\}.
\end{align*}
Applying this with $g:=1$, we obtain 
$$\Gamma^+(1_0)=\{r_1~|~r\in R\cup\{1\}\} \textrm{ and }\Gamma^-(1_0)=\{(l^{-1})_1~|~l\in L\cup\{1\}\}.$$
Since $L\subseteq G\setminus (R^{-1}\cup\{1\})$, we have $(R\cup\{1\})\cap (L^{-1}\cup\{1\})=\{1\}$ and hence $$\Gamma^+(1_0)\cap \Gamma^-(1_0)=\{1_1\}.$$ 
With a similar argument, we have $\G^+(1_1)\cap \Gamma^-(1_1)=\{1_0\}$. Now, since $G$ is transitive on $G_0$ and on $G_1$, we deduce~\eqref{eqle:0}.
In particular, each automorphism of $\G$ fixing $g_i$ must fix also $g_{1-i}$.

Let $A:=\Aut(\G)$ and let $A^{+}$ be the subgroup of $A$ fixing $G_0$ and $G_1$ setwise. Clearly, $|A:A^{+}|\leq 2$.
Observe that each element $\varphi$ of $A^+$ is uniquely determined by a pair $(\varphi_0,\varphi_1)$ of permutations of $G$, where $\varphi_0$ and $\varphi_1$ are defined by the rules $(g^{\varphi_0})_0=(g_0)^\varphi$ and $(g^{\varphi_1})_1=(g_1)^\varphi$, for each $g\in G$. From~\eqref{eqle:0}, we deduce that, for each $\varphi\in A^+$, we have $\varphi_0=\varphi_1$ and hence, using Notation~\ref{notation:1}, every element of $A^+$ is of the form $\phi'$, for some $\phi\in \Sym(G)$.

Let $\phi'\in A^+$, for some $\phi\in \Sym(G)$. By Lemma~\ref{lem=automorphism}, $\phi$ induces an automorphism of $\Cay(G,R\cup\{1\})$ and hence $\phi\in \Aut(\Cay(G,R\cup\{1\}))=\Aut(\Cay(G,R))=G$, because $\Cay(G,R)$ is a DRR. Therefore $A^{+}\leq G$ and hence $A^{+}=G$. This proves~\eqref{eqle:1} and~\eqref{eqle:3}.

Suppose there exists $\a\in \Aut(G)$ with $R^{\a}= L$. Then the mapping $\sigma_{\a,1}$ defined in~\eqref{eq:3} is an automorphism of $\G$ interchanging $G_0$ and $G_1$. Hence $A=\langle G,\sigma_{\a,1}\rangle>G$ and $\G$ is not a HDR.
 Conversely, suppose  $\G$ is not a HDR. Since $A^+=G$ and  $|A:A^+|\le 2$, we deduce $|A:A^+|=2$, $\G$ is vertex transitive and $G\unlhd A$. In particular, there exists $\phi\in A$ with $1_0^{\phi}=1_1$. From~\eqref{eqle:0}, we deduce $1_1^\phi=1_0$. As $\phi\in A=\norm A G$,  by Proposition~\ref{prop=normalizer},
there exist $y\in G$ and $\a\in \Aut(G)$ with $\phi=\sigma_{\a,y}$. Now, $1_0=1_1^\phi=1_1^{\sigma_{\a,y}}=y_0$ and hence $y=1$.
Furthermore, the definition of $\sigma_{\a,y}$ in~\eqref{eq:3} gives $(R\cup\{1\})^{\a}=y^{-1}(L\cup\{1\})=L\cup\{1\}$ and  hence $R^\a=L$. Now,~\eqref{eqle:2} is also proven.
\end{proof}

\begin{lem}\label{lem=oddorder}
Let $G$ be a finite group of order at least $4$ admitting a $\mathrm{DRR}$.
Then $G$ has a subset $R$ with $\Cay(G,R)$ a $\mathrm{DRR}$, $1\notin R$
and $|R|<(|G|-1)/2$.
\end{lem}
\begin{proof} Let $R$ be a subset of $G$ of cardinality as small as possible with $\Cay(G,R)$ a DRR.
Since $$\Aut(\Cay(G,R\cup\{1\}))=\Aut(\Cay(G,R))=G,$$
we have $1\notin R$. Similarly, since $$\Aut(\Cay(G,G\setminus(R\cup\{1\})))=\Aut(\Cay(G,R))=G,$$ we have $|R|\le |G\setminus(R\cup\{1\})|$, that is, $1\leq |R|<|G|/2$. If $|G|$ is even, then $|R|<(|G|-1)/2$. Therefore, we may assume $|G|$ is odd and $|G|\ge 5$. In particular, $G$ is solvable by
the Odd Order Theorem~\cite{FeitThompson}.

If $G$ is cyclic (generated by $a$ say), then $\Cay(G,\{a\})$ is a directed cycle. Thus $\Cay(G,\{a\})$ a DRR over $G$ and $1= |\{a\}|<(|G|-1)/2$.

Suppose $G$ is not cyclic. Let $M$ be a maximal normal subgroup of $G$. As  $G$ is solvable, $G/M$ is cyclic of order $p$, for some odd prime $p$. Let $g\in G\setminus M$ and observe that $$G=\langle M,g\rangle.$$

Assume $M\cong \mz_3^2$. Then $G=\langle a,b,g\rangle$ with $o(a)=o(b)=3$, $ab=ba$
and $p$ dividing $o(g)$. From~\cite[Lemma~$3.4$]{Babai} and from the proof of~\cite[Lemma~$3.1$]{Babai}, $G$ has a subset $R$ with $\Cay(G,R)$ a DRR, $1\notin R$ and $|R|=9$. Clearly, $|R|=9<(|G|-1)/2$, because $|G|=9p\ge 27$.

Assume $M\ncong \mz_3^2$. By Proposition~\ref{prop=TRR}, $M$ has a subset $S$ such that
$\Cay(M,S)$ is a TRR. In particular, $|S|=(|M|-1)/2$ and $S\cap S^{-1}=\emptyset$.
Let $R:=S\cup \{g\}$, let $\Sigma:=\Cay(G,R)$ and let $B:=\Aut(\Sigma)$. For every $s\in S$, neither $(g,s)$ nor $(s,g)$
is an arc of $\Sigma$ and, for every $s_1,s_2\in S$,  exactly one of $(s_1,s_2)$ and $(s_2,s_1)$
is an arc of $\Sigma$. Therefore, $g$ is the unique isolated vertex in the neighbourhood of $1$ in $\Sigma$.
Then, the vertex stabiliser $B_1$ fixes $g$ and fixes $S$ setwise. Therefore, $B_1$ fixes $M=S\cup S^{-1}\cup \{1\}$ setwise and 
hence $B_1$ induces a group of automorphisms on $\Sigma[M]$ (the subgraph induced by $\Sigma$ on $M$).
Since $\Sigma[M]= \Cay(M,S)$ is a TRR, we deduce $B_1=1$ and hence $\Sigma$ is a DRR over $G$ with $|R|=(|M|-1)/2+1<(|G|-1)/2$.
\end{proof}

\begin{proof}[Proof of Theorem~$\ref{theo=main}$] We divide the proof in various cases.

\smallskip

\noindent\textsc{Case 1: }$G$ has no DRR.

\smallskip

\noindent By Lemma~\ref{lem=noDRR},
$G$ has a HDR except when $G$ is isomorphic to  $\mz_2^2$ or $\mz_2^3$.~$_\blacksquare$

\smallskip

For the rest of the proof, we may suppose that $G$ admits a DRR.

\smallskip

\noindent\textsc{Case 2: }$G$ is a elementary abelian $2$-group, that is, $G\cong\mz_2^m$, for some $m\ge 0$.

\smallskip

\noindent By Proposition~\ref{prop=DRR}, $m\in \{0,1\}$ or $m\geq 5$. A direct inspection shows that $\mz_2^0=\mz_1$ and $\mz_2^1=\mz_2$  admit no HDR. In particular, we may suppose that $G=\langle a_1,\ldots,a_m\rangle$ with $m\ge 5$.

When $m=5$, a computation with {\sc Magma} shows that $$\Haar(G,\{1,a_1,a_2,a_3, a_4, a_1a_2, a_5\},\{1,a_1,a_3, a_2a_4, a_1a_2a_3, a_2a_3a_4, a_5\})$$  is a HDR. Suppose then $m\ge 6$ and let $$R:=\{a_1,a_2,\ldots,a_m\}\cup\{a_1a_{2}, a_2a_3,\ldots, a_{m-1}a_m\}\cup
 \{a_1a_2a_{m-2}a_{m-1}, a_1a_2a_{m-1}a_m\}.$$
By \cite{Imrich1}, the Cayley graph $\Cay(G,R)$ is a GRR over $G$ with $|R|=m+(m-1)+2=2m+1$.
Let $H:=\langle a_2,\ldots,a_m\rangle$ and
observe that $$|H\setminus R|=2^{m-1}-(2m-2)>2m+1,$$ because $m\ge 6$. Therefore, there exists a subset $L\subseteq H\setminus (R\cup\{1\})\subseteq G\setminus(R\cup\{1\})=G\setminus (R^{-1}\cup\{1\})$ with $|L|=|R|$.

Let $\G:=\Haar(G,R\cup\{1\},L\cup\{1\})$. Since $\langle L\rangle\neq G$ and $\langle R\rangle=G$, we have $R^{\a}\neq L$ for each $\a\in \Aut(G)$. In particular, Lemma~\ref{lem=A+} gives that $\G$ is a HDR.~$_\blacksquare$

\smallskip

In what follows, we assume $G$ is not an elementary abelian $2$-group and hence $G$ has an element of order at least $3$.

\smallskip

\noindent\textsc{Case 3: }$G$ is cyclic of order $3$.

\smallskip

\noindent  An easy inspection shows that $G$  admits no HDR.~$_\blacksquare$

\smallskip

For the remaining cases, from Lemma~\ref{lem=oddorder}, we see that $G$ admits a DRR $\Cay(G,R)$  with $1\notin R$ and $1\le |R|<(|G|-1)/2$. We partition the set $R$ into two subsets. We let $J:=\{x\in R\mid x^{-1}\notin R\}$ and $K:=R\setminus J$. Observe that $R\setminus J=K$ is inverse-closed, that is, $K^{-1}=\{x^{-1}\mid x\in K\}=K$. Summing up,
\begin{equation*}%\label{equation}
R=J\cup K,\,\,R\cap R^{-1}=K=K^{-1} \textrm { and }J\cap J^{-1}=J\cap K=\emptyset.
\end{equation*}

\smallskip

\noindent\textsc{Case 4: }There exists a subset  $L$ of $G\setminus (R^{-1}\cup\{1\})$ with $|L|=|R|$ and with $R^\a\ne L$, for every $\a\in \Aut(G)$.

\smallskip

\noindent By Lemma~\ref{lem=A+}, $\Haar(G,R\cup\{1\},L\cup\{1\})$ is a HDR.~$_\blacksquare$

\smallskip

For the rest of the proof, we may suppose that, for every subset  $L$ of $G\setminus (R^{-1}\cup\{1\})$ with $|L|=|R|$, there exists $\a\in \Aut(G)$ with $R^\a= L$.

 Let $$H:=G\setminus (\{1\}\cup R\cup R^{-1}).$$
Observe that $G=\{1\}\cup (R\cup R^{-1})\cup H$ is a partition of $G$ and 
\begin{align*}
|H|&=|G|-1-|R\cup R^{-1}|=|G|-1-(|R|+|R^{-1}|-|R\cap R^{-1}|)=|G|-1-(2|R|-|K|).
\end{align*}
Since $2|R|<|G|-1$, we deduce $|H|>|K|$.

\smallskip

\noindent\textsc{Case 5: }There exists $x\in H$ with $o(x)\geq3$.

\smallskip

\noindent Let $U$ be any subset of $H$ with $x\in U$ and $x^{-1}\notin U$ (observe that this is possible because $|H|>|K|$) and let $L:=J\cup U$. Then $|L|=|R|$ and $L\subseteq G\setminus (R^{-1}\cup\{1\})$. Since 
\begin{align*}
|\{y\in R\mid y^{-1}\notin R\}|&=|J|,\\
|\{y\in L\mid y^{-1}\notin L\}|&\ge |J\cup\{x\}|>|J|,
\end{align*} there is no automorphism $\a$ of $G$ with $R^{\a}= L$, which is a contradiction.~$_\blacksquare$ 

\smallskip

\noindent\textsc{Case 6: }No element in $H$ as order at least $3$, that is, each element in $H$ has order $2$.

\smallskip

\noindent Suppose that $K$ contains an element $x$ having order at least $3$. Let $U$ be any subset of $H$ with $|U|=|K|$ and let $L:=J\cup U$. Then $|L|=|R|$ and $L\subseteq G\setminus (R^{-1}\cup\{1\})$. No element in $J$ has order $2$ and hence $$|\{y\in R\mid o(y)=2\}|= |\{y\in K\mid o(y)=2\}|\le |K\setminus\{x\}|=|K|-1.$$
On the other hand, $\{y\in L\mid o(y)=2\}=U$ and hence $|\{y\in L\mid o(y)=2\}|=|U|=|K|$. Therefore, there is no automorphism $\a$ of $G$ with $R^{\a}= L$, which is a contradiction.

Suppose that every element in $K$ has order $2$.
Since $G$ is not an elementary abelian $2$-group and $$G=(R\cup R^{-1})\cup H\cup \{1\}=J\cup J^{-1}\cup K\cup H\cup \{1\},$$ we have $J\neq \emptyset$.
Let $x\in J$, let $U$ be any subset of $H$ with $|U|=|K|+1$ (observe that this is possible because $|H|>|K|$) and let $L:=U\cup (J\setminus\{x\})$. Then $|L|=|R|$ and $L\subseteq G\setminus (R^{-1}\cup\{1\})$. However, since $L$ has more involutions than $R$, there is no automorphism $\a$ of $G$ with $R^{\a}= L$, which is our final contradiction.
\end{proof}

\medskip

\f {\bf Acknowledgement:} This work was supported by the National Natural Science Foundation of China (11571035, 11231008, 11271012) and by the 111 Project of China (B16002).


\begin{thebibliography}{99}
\bibitem{Arezoomand}
M. Arezoomand, B. Taeri, Normality of 2-Cayley digraphs, Discrete Math. 338 (2015), 41--47.

\bibitem{Babai}
L. Babai, Finite digraphs with given regular automorphism groups,
Period. Math. Hungar. 11 (1980), 257--270.

\bibitem{BabaiI}
L. Babai, W. Imrich, Tournaments with given regular group,
Aequationes Math. 19 (1979), 232--244.

\bibitem{magma}
W. Bosma, C. Cannon, C. Playoust, The MAGMA algebra system I: The user language, J. Symbolic Comput. 24 (1997), 235--265.


\bibitem{DSV}E.~Dobson, P.~Spiga, G.~Verret, Cayley graphs on abelian groups, \textit{Combinatorica} \textbf{36} (2016), 371--393.
\bibitem{DFS}J.~-L.~Du, Y.~-Q.~Feng, P.~Spiga, A classification of the graphical $m$-semiregular representations of finite groups, \textit{J. Combin. Theory Ser. A}, to appear.
\bibitem{DFS1}J.~-L.~Du, Y.~-Q.~Feng, P.~Spiga, On the existence and the enumeration of bipartite regular representations of Cayley graphs over abelian groups, submitted.
\bibitem{Duxu}
S.~F.~Du, M.~Y.~Xu, A classification of semi-symmetric graphs of order $2pq$,
Comm. Algebra 28  (2000), 2685--2715.

\bibitem{Estelyi}
I. Est\'elyi, T. Pisanski, Which Haar graphs are Cayley graphs, Electron J. Combin.
23 (2016) $\sharp$P3.10.

\bibitem{FeitThompson}
W.~Feit, J.~G.~Thomphson, Solvability of groups of odd order, Pacific J. Math. 13 (1963), 755--1029.

\bibitem{Godsil}
C.~D.~Godsil, GRR's for non-solvable groups, in Algebraic Methods in Graph theory (Proc. Conf. Szeged 1978 L. Lov\'asz and V.~T.~S\'os, eds), Coll. Math. Soc. J. Bolyai 25,
North-Holland, Amsterdam, 1981, 221--239.

%\bibitem{Godsil2}
%C.D. Godsil, Tournaments with prescribed regular automorphism group,
%Aequationes Math. 30 (1986) 55--64.

\bibitem{Hetzel}
D. Hetzel,  \"Uber regul\"are graphische Darstellung von aufl\"osbaren Gruppen, Technische Universit\"at, Berlin, 1976.

\bibitem{Hujdurovic}
A. Hujdurovi$\acute{c}$, K. Kutnar, D. Maru$\check{s}$i$\check{c}$, On normality of $n$-Cayley graphs, Appl. Math. Comput. 332 (2018), 469--476.

\bibitem{Imrich1}
W. Imrich, Graphs with transitive Abelian automorphism group in Combinatorial
Theory and Its Applications, Coll. Soc. Janos Bolyai 4, Balatonfued,
Hungary, (1969), 651--656.

\bibitem{Imrich}
W. Imrich, Graphical regular representations of groups odd order, in: Combinatorics, Coll. Math. Soc. J\'anos. Bolayi 18 (1976), 611--621.

\bibitem{ImrichWatkins}
W. Imrich,  M.E. Watkins, On graphical regular representations of cyclic extensions of groups, Pac. J. Math. 55 (1974), 461--477.

\bibitem{MSV}J.~Morris, P.~Spiga, G.~Verret, Automorphisms of Cayley graphs on generalised
dicyclic groups, European J. Combin. 43 (2015), 68--81.
%\bibitem{ImrichWatkins2}
%W. Imrich, M. E. Watkins, On automorphism groups of Cayley graphs,
%Period. Math. Hungar. 7 (1976) 243--258.

\bibitem{MorrisSpiga1}
J. Morris, P. Spiga, Every finite non-solvable group admits an oriented regular representation, J. Combin. Theory Ser. B 126 (2017), 198--234.

\bibitem{MorrisSpiga3}J. Morris, P.~Spiga, Classification of finite groups that admit an oriented regular
representation, Bulletin of the London Math. Soc. (2018), 811--831.

\bibitem{NowitzWatkins1}
L.~A.~Nowitz, M.~E.~Watkins, Graphical regular representations of non-abelain groups, $I$,
Canad. J. Math. 24 (1972), 994--1008.

\bibitem{NowitzWatkins2}
L.~A.~Nowitz, M.~E.~Watkins, Graphical regular representations of non-abelain groups, $II$,
Canad. J. Math. 24 (1972), 1009--1018.

\bibitem{Spiga}
P. Spiga, Finite groups admitting an oriented regular representation,
J. Combin. Theory Ser. A 153 (2018), 76--97.
\bibitem{Spiga2}
P. Spiga, On the Existence of Frobenius Digraphical Representations, Electron. J. Comb. 25 (2018), $\sharp$P2.6

\bibitem{Spiga3}
P. Spiga, Cubic graphical regular representations of finite non-abelian simple groups, Commu. Algebra,
46 (2018), 2440--2450.

\bibitem{Watkins}
M. E. Watkins, On the action of non-abelian groups on graphs, J. Combin. Theory 11 (1971), 95--104.
\bibitem{Xiaf}
B.~Z.~Xia, T.~Fang, Cubic graphical regular representations of $\mathrm{PSL}_2(q)$, Discrete Math. 339 (2016), 2051--2055.

\end{thebibliography}
\end{document}